%% file: bac-oa_arX2.tex
\documentclass[11pt]{article}

\title{\vskip-1.0em\sc Singly generated operator algebras satisfying weakened versions of amenability}
\author{Yemon Choi}
\date{2nd July 2012}

\usepackage{sectsty}
\allsectionsfont{\sc}

\input{bac-oa_mac2b}

\begin{document}
\maketitle

\begin{abstract}
We construct a singly generated subalgebra of $\Cpct(\cH)$ which is non-amenable, yet is boundedly approximately contractible. The example embeds into a homogeneous von Neumann algebra.
We also observe that there are singly generated, biflat subalgebras of finite Type I von Neumann algebras, which are not amenable (and hence are not isomorphic to $\Cst$-algebras).
Such an example can be used to show that a certain extension property for commutative operator algebras, which is shown in \cite{YC_acoaf} to follow from amenability, does not necessarily imply amenability.

\end{abstract}

\begin{section}{Introduction}
In his pioneering 1972 monograph \cite{BEJ_CIBA}, which formally introduced the notion of amenability for Banach algebras, Johnson observed that commutative $\Cst$-algebras are amenable.
% The full characterization of those $\Cst$-algebras whose underlying Banach algebras are amenable took much longer, and was built on the work of several authors, Connes and Haagerup in particular. See \cite{RunLec} for a fuller account.
The problem of characterizing the commutative, amenable operator algebras -- that is, the norm-closed subalgebras of $\Bdd(\cH)$, $\cH$ a Hilbert space, which are amenable as Banach algebras -- remains unsolved, even in the commutative case. There are some partial results which suggest that every commutative, amenable operator algebra is isomorphic as a Banach algebra to a commutative $\Cst$-algebra: see recent work of Marcoux~\cite{Marcoux_JLMS08} and the present author~\cite{YC_acoaf}.

Like many other concepts which have been much studied, amenability has spawned several weaker versions, some studied in more depth than others. In trying to assess the relative worth, importance or interest of these notions, there are two natural criteria to consider. Firstly, if a Banach algebra satisfies one of these weaker versions of amenability, does that tell us something about its internal structure (e.g.~properties of its closed ideals) or about the behaviour of modules/representations of the algebra?
Secondly, is there a good supply of examples which \emph{are non-amenable} yet satisfy the weaker version?

It is this second question, in the context of operator algebras, which motivates the present note.
 We consider two of these weaker versions of amenability, namely bounded approximate contractibility and biflatness, in the context of commutative operator algebras. In each case, we construct an explicit example which satisfies that property, and is singly generated and semisimple, but is non-amenable; the \bac\ example is contained in $\Cpct(\cH)$.

Our examples can be embedded inside finite von Neumann algebras of \typone\ (more precisely, inside countable products of matrix algebras).
 This shows that in the results of \cite{YC_acoaf}, one cannot replace amenability with biflatness, nor with bounded approximate contractibility.
 In passing, we take the opportunity to show that certain technical arguments from~\cite{YC_acoaf} have rather sensitive hypotheses.
We remark that these examples therefore embed inside \typtwoone\ algebras, since any finite Type~I algebra with separable predual can be embedded into the hyperfinite $\twoone$ factor (J.~Peterson, personal communication, \cite{MO85878}).

\end{section}

\begin{section}{General preliminaries}\label{s:prelim}
We refer the reader to \cite{BonsDunc} for basics on Banach algebras, and to \cite[\S43]{BonsDunc} in particular for the basic definitions and characterizations of amenable Banach algebras.

The examples we will construct in this article are each generated by a countable set of pairwise-orthogonal idempotents. We note that this ensures they are singly generated as Banach algebras; this is a general fact, which is not restricted to the setting of subalgebras of $\Bdd(\cH)$.

\begin{lem}\label{l:singly-generated}
Let $A$ be a commutative Banach algebra generated by a countable set of pairwise-orthogonal idempotents. Then $A$ is singly generated as a Banach algebra.
\end{lem}

This is presumably a standard result, but for convenience, we provide a full proof, inspired by a similar argument in \cite[Proposition 1.1]{AzSh_JOT93} for certain subalgebras of $\Bdd(\cH)$.

\begin{proof}
Enumerate the given set of idempotents as $(e_j)_{j\geq 1}$, and fix a strictly decreasing sequence of strictly positive reals $\lm_1>\lm_2>\dots$ with the property that $\sum_{j\geq 1} \lm_j \norm{e_j}<\infty$.
Set $b=\sum_{j\geq 1} \lm_j e_j$, and let $B$ be the norm-closed subalgebra of $A$ generated by~$b$.
We claim that $e_n\in B$ for all $n\geq 1$, which will clearly imply $B=A$ since $A$ is generated by the set $\{e_j\st j\in\Nat\}$.

The proof is by strong induction on $n$. We start by noting that for all $r\in\Nat$, we have $b^r = \sum_{j\geq 1} \lm_j^r e_j$,
 the sum converging absolutely. As $0< \lm_j \leq \lm_2 < \lm_1$ for all $j\geq 2$, it follows that
\[ \norm{e_1 - (\lm_1^{-1}b)^r} \leq \sum_{j\geq 2} \left( \frac{\lm_j}{\lm_1}\right)^r\norm{e_j} \leq \frac{1}{\lm_1} \left(\frac{\lm_2}{\lm_1}\right)^{r-1}\sum_{j\geq 2} \lm_j \norm{e_j} \to 0 \quad\text{as $r\to \infty$.} \]
Thus $e_1 \in B$, and so the claim holds for $n=1$.

Now suppose the claim holds for all $n\in \{1,\dots, m-1\}$ for some $m\geq 2$. Let
\[ b_m = b- \left(\sum_{j=1}^{m-1} \lm_j e_j\right) = \sum_{j\geq m} \lm_je_j;\]
by the inductive hypothesis, $b_m \in B$. For all $r\in\Nat$ we have $b_m^r = \sum_{j\geq m} \lm_j^r e_j$,
the sum converging absolutely. As $0< \lm_j \leq \lm_{m+1} < \lm_m$ for all $j\geq m+1$, a similar argument to the one used above shows that $e_m= \lim_r (\lm_m^{-1}b_m)^r \in B$. So the claim holds for $n=m$, and this completes the inductive step.
\end{proof}

% It has been known since the early days of work on amenable Banach algebras that that this corresponds to a kind of splitting property for modules. \fbox{\textsf{Reference: \cite{CuLo_amen}}}

If $\fA$ is a Banach algebra, isomorphic to a $\Cst$-algebra, then the set of central idempotents in $\fA$ will be a bounded subset of $\fA$. (This is because the centre of a $\Cst$-algebra is itself a commutative $\Cst$-algebra, and hence all its non-zero idempotent elements have norm~$1$.) In particular, a commutative operator algebra which is isomorphic to a $\Cst$-algebra will have uniformly bounded idempotents. Gifford observed that the same will be true for any amenable, commutative operator algebra. More precisely, we have the following result (in whose statement the prime denotes the commutant of a subset of $\Bdd(\cH)$).

\begin{thm}[Gifford, \cite{Gifford}]\label{t:Gifford-trick}
If $\fA$ is an amenable operator algebra, then the set of idempotents in $\fA''\cap\fA'$ is norm-bounded.
\end{thm}

Gifford actually proved something more general, namely that the conclusion of Theorem~\ref{t:Gifford-trick} holds whenever $\fA$ is an operator algebra with the so-called \dt{total reduction property}. Since we shall not discuss the total reduction property in this article, we refer the reader to \cite{Gifford} for further details, and provide a more direct proof of Theorem~\ref{t:Gifford-trick} as follows.

\begin{proof}
Regard $\fA$ as a closed subalgebra of some $\Bdd(\cH)$. Since the unitization of an amenable Banach algebra is amenable, we may assume without loss of generality that $\fA$ contains the identity operator $I$.

Let $(\Delta_\al)$ be a bounded approximate diagonal for $\fA$. Define $T: \fA\ptp\fA \to \Bdd(\Bdd(\cH))$ by $T(c\tp d)(x) = cxd$, and let $E$ be a point-to-weak$^*$ cluster point of the net $T(\Delta_\al)\subset \Bdd(\Bdd(\cH))$. As $(\Delta_\al)$ is a bounded approximate diagonal for $\fA$, it follows from routine estimates and convergence arguments that the following properties hold:
\begin{YCnum}
\item \label{li:into-A'}
$E(x) \in\fA'$ for all $x\in\Bdd(\cH)$;
\item \label{li:fixing-A'}
$E(u)=u$ for all $u\in\fA'$\/;
\item \label{li:A'-bimodule}
$E(uxv) = uE(x)v$ for all $u,v\in \fA'$ and all $x\in \Bdd(\cH)$.
\end{YCnum}

Let $e\in \fA''\cap\fA'$ be an idempotent, and let $p$ be the orthogonal projection from $\cH$ onto the closed subspace $e\cH$. We have $ep=p$ and $pe=e$. Therefore,
\[ \begin{aligned}
E(p) & = E(ep) \\
 & = e E(p) \quad\text{(by \ref{li:A'-bimodule} and $e\in\fA'$)} \\
 & = E(p) e \quad\text{(by \ref{li:into-A'} and $e\in\fA''$)} \\
 & = E(pe) \quad\text{(by \ref{li:A'-bimodule} and $e\in\fA'$)} \\
 & = E(e) \\
 & = e  \quad\text{(by \ref{li:fixing-A'} and $e\in \fA'$.}
\end{aligned} \]
In particular, $\norm{e}=\norm{E(p)} \leq\norm{E}$. Since $E$ is independent of the choice of $e$, we are done.
\end{proof}

\begin{rem}
The proof just given requires little background from the theory of amenability, but may seem somewhat unmotivated. For a more conceptual but less self-contained approach, the reader is encouraged to consult Gifford's original article \cite{Gifford}.
\end{rem}

\begin{cor}\label{c:unbdd-idem}
Let $\fA$ be a commutative operator algebra which contains an unbounded family of idempotents. Then $\fA$ is not amenable.
% Also not similar to $\Cst$ -- this either follows from knowing that commutative $\Cst$-algebras are amenable, or by a direct observation that in a $\Cst$-algebra all non-zero idempotents have norm~$1$.
\end{cor}

\end{section}

\begin{section}{Bounded approximate contractibility}
In \cite{GhL_genam1}, Ghahramani and Loy began the systematic study of certain ``approximate'' versions of amenability. These have since been pursued in various directions by several different groups of authors; see \cite{Zhang_survey} for an overview of some of the results to date.

In this article, we are only concerned with one such variant, which we now briefly describe. Given a Banach algebra $A$ and a Banach $A$-bimodule $X$ and $\xi\in X$, we denote by $\ad{\xi}$ the inner derivation $a\mapsto a\cdot\xi-\xi\cdot a$.

\begin{dfn}
A Banach algebra $A$ is \dt{\bac} if for each Banach $A$-bimodule $X$ and each continuous derivation $D:A\to X$, there exists a net $(\xi_i)\subset X$, not necessarily bounded, such that the net $(\ad{\xi_i})$ is norm bounded (as a subset of $\Bdd(A,X)$) and converges in the strong operator topology of $\Bdd(A,X)$ to~$D$.
\end{dfn}

Note that since there are no non-zero, inner derivations from a Banach algebra $A$ into a symmetric $A$-bimodule, every \emph{commutative}, \bac\ Banach algebra is weakly amenable.

\begin{dfn}\label{d:mbad}
Let $A$ be a Banach algebra, and let $\pi:A\ptp A \to A$ denote the linearized multiplication map. A \dt{\mbad} for $A$ is a net $(M_i)\subset A\ptp A$ such that
\begin{YCnum}
\item for each $a\in A$, $\lim_i \norm{a\pi(\Delta_i)-a} =0$;
\item for each $a\in A$, $\lim_i \norm{a\cdot \Delta_i-\Delta_i\cdot a} = 0$;
\item there exists a constant $C>0$ such that $\sup_i \norm{a\cdot \Delta_i-\Delta_i\cdot a} \leq C\norm{a}$ for all $a\in A$.
\end{YCnum}
\end{dfn}

\begin{lem}\label{l:finding-bac}
Let $A$ be a Banach algebra. Suppose there exists a net $(\Delta_i)\subseteq A\tp A$ with the following properties:
\begin{itemize}
\item[{\rm(a)}] $\pi(\Delta_i)$ is a \emph{central, bounded} approximate identity for $A$;
\item[{\rm(b)}] $(\Delta_i)$ is a \mbad\ for $A$.
\end{itemize}
Then $A$ is \bac.
\end{lem}

 (One can weaken condition (a), but the proof becomes technically more demanding, and the version given here will be enough for our purposes.)

\begin{proof}
As observed in \cite[Proposition 2.2]{CGZ_JFA09}, it suffices to show that the unitization $\fu{A}$ has a \mbad. Let $u_i=\pi(\Delta_i)$, and define
\[ M_i = 2\Delta_i - u_i\cdot\Delta_i + (\id-u_i)\tp(\id-u_i). \]
We claim that $(M_i)$ is a \mbad\ for $\fu{A}$.

Firstly, let $\fu{\pi}:\fu{A}\ptp \fu{A} \to \fu{A}$ be the product map, which clearly satisfies $\fu{\pi}\vert_{A\ptp A} = \pi$. Then, using the identity $\pi(u\cdot x)=u\pi(x)$ for all $u\in A$ and $x\in A\ptp A$, a direct calculation yields $\fu{\pi}(M_i)=\id$ for all~$i$. So condition~(i) of Definition~\ref{d:mbad} is satisfied.

Secondly, let $K=\sup_i \norm{u_i} < \infty$, and observe that as $(u_i)$ lies in the centre of $A$, we have for each $a\in A$
\[
\begin{aligned}
a\cdot M_i - M_i\cdot a
 & = 
2(a\cdot\Delta_i -\Delta_i\cdot a)  - u_i\cdot (a\cdot \Delta_i - \Delta_i\cdot a)  \\
& \quad + (a- au_i)\tp (\id-u_i) - (\id-u_i)\tp(a-au_i).
\end{aligned}
\]
Hence for each $a\in A$ and $\lm\in\Cplx$,
\[ \begin{aligned}
\norm{(a+\lm\id)\cdot M_i - M_i \cdot (a+\lm\id)}
 & =  \norm{a\cdot M_i - M_i \cdot a} \\
 & \leq (2+K)\norm{a\cdot\Delta_i-\Delta_i\cdot a} + 2(1+K)\norm{a-au_i}.
\end{aligned}
\]
Thus condition (ii) of Definition~\ref{d:mbad} is satisfied. Finally: by assumption, there exists $C>0$ such that $\norm{a\cdot\Delta_i-\Delta_i\cdot a}\leq C\norm{a}$ for all $a\in A$. Hence,
\[ \begin{aligned}
 \sup_i \norm{ (a+\lm\id)\cdot M_i - M_i \cdot (a+\lm\id) }
& \leq (2+K)C\norm{a}+2(1+K)^2\norm{a} \\
& \leq ((2+K)C + 2(1+K)^2)\ \norm{a+\lm \id}.
\end{aligned} \]
Thus condition (iii) of Definition~\ref{d:mbad} is satisfied.
This completes the proof of the claim, and hence of the lemma.
\end{proof}

To construct explicit examples, we use the following result. It is a slightly more abstract version of known results for Banach sequence algebras (cf.~\cite[Theorem 4.4]{GhLZ_genam2} or \cite[Corollary 3.5]{DLZ_St06}), and it extends \cite[Example 4.6]{GhLZ_genam2}, although our approach is slightly different from the argument there.

\begin{prop}\label{p:genam-eg}
Let $A$ be a Banach algebra, containing a sequence $(e_n)_{n\geq 1}$ with the following properties:
\begin{YCnum}
\item\label{li:absorb}
 $e_me_n = e_{\min(m,n)} = e_ne_m$ for all $m,n\in\Nat$;
\item\label{li:total}
 the set $\{ e_n \st n \in \Nat\}$ has dense linear span in $A$;
\item\label{li:BAI}
 there is a subsequence $n(1)<n(2)<\dots$ such that $(e_{n(j)})_{j\geq 1}$ is a bounded approximate identity for~$A$.
\end{YCnum}
Then $A$ is \bac.
\end{prop}

\begin{proof}
This is similar to the proof of \cite[Theorem 6.1]{CGZ_JFA09}.
For $n\geq 2$, define
\[ \Delta_n = e_1 \tp e_1 + \sum_{j=2}^n (e_j-e_{j-1})\tp(e_j-e_{j-1}). \]
We then have the following identities:
\begin{align}
\label{eq:needed1}
\pi(\Delta_n) & =e_n \quad\text{for all $n$;} \\
\label{eq:needed2}
a\cdot\Delta_n & = \Delta_n\cdot a \quad\text{for all $n$ and all $a\in A$.}
\end{align}
The identity \eqref{eq:needed1} can be shown by direct calculation, using property~\ref{li:absorb}. The identity \eqref{eq:needed2} is true for $a=e_m$ ($m$ arbitrary); this is another direct calculation using~\ref{li:absorb}, which is most easily done by treating the cases $m\leq n$ and $m>n$ separately. Hence, by linearity and continuity (using property~\ref{li:total}), this identity holds for all $a\in A$, as claimed.

It follows immediately from \eqref{eq:needed1}, \eqref{eq:needed2} and property~\ref{li:BAI} that:
\begin{itemize}
\item[(a)] $\pi(\Delta_{n(j)})$ is a bounded approximate identity for~$A$; and
\item[(b)] the sequence $(\Delta_{n(j)})$ is a \mbad\ for $A$.
\end{itemize}
Therefore $A$ is \bac, by Lemma~\ref{l:finding-bac}.
\end{proof}

\begin{rem}
The proof of Proposition~\ref{p:genam-eg} shows slightly more, namely that $A$ is not only \bac, but also \dt{pseudo-contractible} in the sense of \cite{GhZ_pseudo}. (There are pseudo-contractible Banach algebras which are not \bac, and vice versa.)
\end{rem}

Now, fix a Hilbert space $\cH$ and a strictly ascending chain of non-zero subspaces
 $\cH_1 \subset \cH_2 \subset \cH_3 \subset \dots$; for each $n\in\Nat$, let
 $p_n$ be the orthogonal projection of $\cH$ onto $\cH_n$.
For each $k\in\Nat$, choose a bounded operator $b_{2k}\in \Bdd(\cH_{2k+1}\ominus\cH_{2k}, \cH_{2k}\ominus \cH_{2k-1})$, such that $\norm{b_{2k}} \to \infty$ as $k\to\infty$, and define a sequence $(e_n)_{n\geq 1} \subset\Bdd(\cH)$ by
\[ e_{2k-1}\defeq p_{2k-1} \quad\text{and}\quad e_{2k}\defeq p_{2k} + b_{2k} (p_{2k+1}-p_{2k}) \quad\text{for $k=1,2,\dots$} \]
Thus, in block matrix form,
\[
e_{2k-1} = \left[ \begin{matrix}
    I & 0 & 0 & 0 \\ 0 & 0 & 0 & 0 \\ 0 & 0 & 0 & 0 \\ 0 & 0 & 0 & 0 \\
  \end{matrix} \right]
\; \begin{matrix}
 \cH_{2k-1} \\ \cH_{2k}\ominus\cH_{2k-1} \\ \cH_{2k+1}\ominus \cH_{2k} \\ \cH\ominus\cH_{2k+1}
\end{matrix}
\quad\text{and}\quad
e_{2k} = \left[ \begin{matrix}
    I & 0 & 0 & 0 \\ 0 & I & b_{2k} & 0 \\ 0 & 0 & 0 & 0 \\ 0 & 0 & 0 & 0 \\
  \end{matrix} \right]
\; \begin{matrix} \cH_{2k-1} \\ \cH_{2k}\ominus\cH_{2k-1} \\
 \cH_{2k+1}\ominus \cH_{2k} \\ \cH\ominus\cH_{2k+1}
\end{matrix}
\]

It is easily checked that for each $n$, we have $e_n^2=e_n$ and
$e_ne_{n+1}=e_n=e_{n+1}e_n$ (consider the cases of odd and even $n$ separately).
The latter property implies, by induction, that $e_me_n=e_{\min(m,n)} = e_ne_m$ for all $m,n\in\Nat$. Since $\norm{e_{2k-1}} = 1$ for all $k$,
 we see that the algebra $\fA = \clin\{ e_n \st n\in\Nat\}$ satisfies the conditions of Proposition~\ref{p:genam-eg}. Hence it is a \bac\ Banach algebra. It is also singly generated, by Lemma~\ref{l:singly-generated}. On the other hand, since
\[ \norm{e_{2k}} \geq \norm{b_{2k}} \to \infty \quad\text{ as $k\to\infty$,} \]
$\fA$ is non-amenable, by Corollary~\ref{c:unbdd-idem}.

In particular, suppose that we take $\cH=\ell^2(\Z_+)$, with its standard o.n.~basis $(\delta_n)_{n\geq 0}$, and define $\cH_k = \lin(\delta_0,\dots ,\delta_k)$, so that each $b_{2j}$ is a scalar. Then since each idempotent $e_n$ has finite rank, $\fA\subset \Cpct(\cH)$; and $\fA$ is singly generated, by Lemma~\ref{l:singly-generated}.
Using that lemma, we thus obtain an example of a compact operator on Hilbert space which generates a non-amenable, \bac\ algebra. (We recall that by \cite{Wil_amenop}, a compact operator on Hilbert space generates an amenable algebra if and only if it is similar to a normal compact operator.)

Furthermore, it is clear that $\fA$ is contained in the subalgebra of $\Bdd(\ell^2(\Z_+))$ formed by the block-diagonal matrices of block size~$2$, i.e.~it embeds into the finite, homogeneous von Neumann algebra $\ell^\infty(\Z_+)\otimes \Mat_2$. We finish this section by noting that one can obtain many examples with the same properties, by choosing different sequences $(\cH_n)$ and $(b_{2k})$. It may be interesting to study how these different examples might differ in their finer structure, or share other common features beyond being \bac, singly and compactly generated, etc.
\end{section}

\begin{section}{Biflatness}
The notion of a biflat Banach algebra is due to Helemski{\u\i}. It should be emphasized that it is not an ad hoc weakening of amenability, obtained by randomly omitting certain conditions and ``seeing what happens''. Rather, it is linked to the notion of (homologically) flat Banach modules over a given Banach algebra, which is itself a key concept in Helemski{\u\i}'s versions of Ext and Tor in the Banach algebraic setting. 

That said, we shall not concern ourselves with the deeper homological implications of biflatness, and we will not give the original homological definition. For that, the reader should consult \cite{Hel_HBTA}. Instead, we take the following, well known characterization of biflatness as a working definition.

\begin{lem}[see~{\cite[Exercise VII.2.8]{Hel_HBTA}}]
A Banach algebra $A$ is biflat if and only if there exists a continuous, linear $A$-bimodule map $\sigma: A \to (A\ptp A)^{**}$ such that $\pi^{**}\sigma=\kp_A$, the natural embeddding $A\to A^{**}$.
\end{lem}

It is easy to see, from this characterization, that:
\begin{YCnum}
\item every amenable Banach algebra is biflat (if $M \subset (A\ptp A)^{**}$ is a virtual diagonal for $A$, define $\sigma(a)=a\cdot M$);
\item a biflat Banach algebra with a bounded approximate identity is amenable
 (if $(u_i)$ is a BAI for $A$, let $M$ be any \wstar-cluster point in $(A\ptp A)^{**}$ of the net $\sigma(u_i)$).
\end{YCnum}

\begin{rem}
Biflat Banach algebras are simplicially trivial, i.e.~if $A$ is biflat then the continuous Hochschild cohomology groups $\cH^n(A,A^*)$ vanish for all $n\geq 1$. In particular, biflat algebras are weakly amenable.
\end{rem}

Let $A$ denote the Banach algebra obtained by equipping $\ell^1$ with pointwise multiplication. $A$~is a standard example of a commutative, semisimple, biflat Banach algebra that has no bounded approximate identity, and hence is non-amenable. Moreover, by Lemma~\ref{l:singly-generated}, it is singly generated as a Banach algebra.

It turns out that there is a continuous algebra homomorphism $\theta:A\to \Bdd(\cH)$ whose range is closed, so that $\theta(A)$ is a singly generated, biflat operator algebra.
% Furthermore: a result of Blecher and Le Merdy tells us that, for any given operator space structure on $\ell^1$ which makes the multiplication map $A\tp A\to A$ completely contractive, there exists a closed subalgebra $\fA\subset \Bdd(\cH)$ and an algebra homomorphism $\theta:A\to\fA$ which is a complete isomorphism of operator spaces. Each such $\fA$ is then an operator algebra that is operator biflat, but not operator amenable.
Moreover, given such an embedding $\theta$, one can construct an embedding of $A$ as a closed subalgebra of a finite, Type I von Neumann algebra.
(The basic idea is as follows. If $(e_n)$ denotes the standard unit basis of $A=\ell^1$, let $A_n=\lin(e_1,\dots,e_n)$, observe that $\theta\vert_{A_n}$ can be viewed as a representation of $A_n$ on a finite-dimensional Hilbert space, and then take the direct product of the representations $\theta \pi_n$ where $\pi_n:A \to A_n$ is the obvious truncation homomorphism.)

The existence of a map $\theta$ with these properties can be shown by combining the following two results: $A$ is a \dt{$Q$-algebra} in the sense of Varopoulos (proved independently by Davie \cite{Davie_Q-alg} and Varopoulos \cite{Var72_Q-alg}, see also Example~18.3 and Theorem~18.7 in \cite{DieJarTon}); and every $Q$-algebra is isomorphic to some closed subalgebra of $\Bdd(\cH)$ (this is a theorem of Cole; see \cite[\S50]{BonsDunc} or \cite[Theorem 18.8]{DieJarTon}). This approach is somewhat indirect, and does not seem to give an explicit description of an embedding.
It is therefore desirable to have an explicit construction of an embedding of $A$ as a closed subalgebra of a product of matrix algebras.
This can be done with the following construction, which was shown to me by M.~de~la~Salle \cite{MO49788},
 and is included here with his kind permission.
The presentation here is paraphrased slightly from his original wording.
 It seems likely that similar embeddings were known previously, but I was unable to find an explicit description in the literature.

\begin{lem}\label{p:delaSalle}
Consider the Hilbert space $\ell^2(\Nat\cup\{\alpha,\om\})$, where $\alpha$ and $\om$ are formal symbols.
There exists a sequence of rank-one idempotents $(E_n)_{n\geq 1} \subset \Bdd(\ell^2(\Nat\cup\{\alpha,\om\}))$, with the following properties:
\begin{YCnum}
\item $\norm{E_n}\leq 3$ for all $n$;
\item $E_jE_k=0$ whenever $j\neq k$;
\item for each $n$, $\ran(E_n)$ and $\ran(E_n^*)$ are both contained in $\lin(e_\alpha, e_\om, e_n)$;
\item\label{li:Ponting} $\norm{\sum_j a_j E_j}\geq \abs{ \sum_j a_j }$ for any $a\in c_{00}$.
\end{YCnum}
\end{lem}

\begin{proof}[Proof (de la Salle, personal communication)]
Let $e_\al,e_\om, e_1,e_2,\dots$ denote the standard basis vectors, and for each $n\in\Nat$ put $x_n = e_\om + e_\al+e_n$, $y_n = e_\om - e_\al + e_n$.
Define $E_n$ by taking $E_n( \xi) = y_n\ip{\xi}{x_n}$. Clearly each $E_n$ is a rank-one operator; direct calculation shows that $E_n^2=E_n$. Properties (i)--(iii) are also easily verified, and~\ref{li:Ponting} follows from observing that $\ip{E_ne_\om}{e_\om} = 1$.
\end{proof}

\paragraph{Constructing the desired embedding.}
Let $\cF$ denote the family of finite, non-empty subsets of $\Nat$.
For each $F\in\cF$ let $\Mat_{F\cup\{\al,\om\}}$
be the algebra of square matrices indexed by $F\cup\{\al,\om\}$, given the usual ($\Cst$-algebra) norm; then if $j \in F$ we can identify $E_j$ and $E_j^*$ with elements of $\Mat_{F\cup\{\al,\om\}}$.

Let $\cM$ be the finite, \typone\ von Neumann algebra $\prod_{F\in\cF} \Mat_{F\cup\{\al,\om\}}$ -- we remark that this has separable predual, since $\cF$ is countable -- and define
$\phi: A \to \cM$ by
\begin{equation}\label{eq:the-embedding}
\phi(\delta_j)_F = \left\{ \begin{aligned} 0 & \quad\text{if $j\notin F$} \\ E_j & \quad\text{if $j\in F$} \end{aligned} \right.
\end{equation}
Clearly $\phi$ is bounded linear with $\norm{\phi}\leq \sup_j \norm{E_j} \leq 3$, and since
\[ \phi(\delta_j)_F\phi(\delta_k)_F = \left\{ \begin{aligned}
\phi(\delta_j) & \quad\text{ if $j=k$,} \\
 0 & \quad\text{ if $j\neq k$,}
 \end{aligned}\right. \]
it follows by continuity that $\phi$ is an algebra homomorphism. To see that $\phi$ is bounded below, we use the estimate
\begin{equation}\label{eq:RCA}\
\sup_{F\in\cF} \abs{ \sum_{j\in F} a_j } \geq \frac{1}{\pi} \norm{a}_1\, .
\end{equation}
(This can be found as \cite[Lemma 6.3]{Rudin_RCA3}. An earlier version of this paper had a weaker estimate, with $4$ instead of $\pi$; my thanks to the referee for providing a reference for the sharper estimate.) We then have, for each $a\in A$,
\[ \begin{aligned}
\norm{\phi(a)}
  = \sup_{F\in\cF} \norm{\phi(a)_F} 
 & = \sup_{F\in\cF} \norm{ \sum_{j\in \Nat} a_j\phi(\delta_j)_F } \\
 & = \sup_{F\in\cF} \norm{ \sum_{j\in F} a_j E_j } \\
 & \geq \sup_{F\in\cF} \abs{\sum_{j\in F} a_j } & \quad\text{(by Proposition~\ref{p:delaSalle}\ref{li:Ponting})} \\
 & \geq \frac{1}{\pi} \norm{a}_1 & \quad\text{(by the inequality \eqref{eq:RCA}).}
 \end{aligned} \]
Thus $\phi$ has closed range, as required.
\end{section}

\begin{section}{An extension result}
As previously mentioned: it has been asked if every amenable commutative subalgebra $\fA\subseteq\Bdd(\cH)$ is isomorphic to a commutative $\Cst$-algebra, which is equivalent to asking if the Gelfand transform $\sG_\fA:\fA\to C_0(\Phi_\fA)$ is an isomorphism of Banach algebras. The answer is positive if $\fA$ is contained in a finite von Neumann algebra $\cM$. A key part of the proof of this in \cite{YC_acoaf} is the following technical result.

\begin{prop}\label{p:acoaf}
Let $\cM$ be a von Neumann algebra with a faithful, finite, normal trace~$\tau$, and let $\fA\subseteq\cM$ be a closed subalgebra which is commutative and amenable. Let $\Phi_\fA$ denote the character space of $\fA$, and let $\sG_\fA:\fA\to C_0(\Phi_\fA)$ denote the Gelfand transform. Then $\sG_{\fA}$ is injective with dense range, and there is a bounded linear map $\theta: C_0(\Phi_\fA) \to L^1(\cM,\tau)$ which extends the inclusion $\fA\to \cM$, in the sense that the diagram
\begin{equation}\label{eq:diagram}
\begin{CD}
\fA & @>{\hbox{inclusion}}>> & \cM \\
@V{\sG_{\fA}}VV & & @VV{\hbox{inclusion}}V \\
C_0(\Phi_\fA) & @>{\theta}>> & L^1(\cM,\tau)
\end{CD}
\end{equation}
commutes.
Here, $L^1(\cM,\tau)$ denotes the completion of $\cM$ with respect to the norm $\norm{x}_{L^1(\tau)}= \sup\{ \abs{\tau(xy)} \st y \in \cM, \norm{y}\leq 1\}$.
\end{prop}

One might wonder if the converse of this proposition holds, perhaps with the addition of some side-conditions on $\fA$. More precisely:

\paragraph{Question.} Let $\cM$ be as in Proposition~\ref{p:acoaf}, and let $\fA\subseteq \cM$ be a closed subalgebra, such that
\begin{YCnum}
\item $\sG_\fA:\fA\to C_0(\Phi_\fA)$ is injective with dense range;
\item there exists a bounded linear map $\theta: C_0(\Phi_\fA)\to L^1(\cM,\tau)$ making the diagram \eqref{eq:diagram} commute.
\end{YCnum}
If $\fA$ furthermore satisfies some weak version of amenability, must it automatically be amenable?

\medskip
If we take ``weak version of amenability'' to be biflatness (and therefore, anything strictly weaker than biflatness), then the answer is negative. Our example will be the algebra $\fA$ obtained as an embedding of $\ell^1$.
Let $\phi$ be the embedding constructed in \eqref{eq:the-embedding}, and let $\sG_A: A \to c_0(\Nat)$ be the Gelfand transform (which is just the usual injection from $\ell^1(\Nat)$ into $c_0(\Nat)$.)
% Clearly it does not extend to a continuous homomorphism $c_0(\Nat)\to \cM$.

\begin{prop}\label{p:extension}
Let $\tau$ be a faithful, normal, finite, tracial state $\tau$ on $\cM$.
Then
\begin{equation}\label{eq:Monty}
\norm{\phi(a)}_{L^1(\tau)} \leq 3 \norm{a}_\infty \quad\text{for all $a\in \ell^1(\Nat)$.}
\end{equation}
Hence there is a bounded linear map $\theta: c_0(\Nat) \to L^1(\tau)$ which makes the diagram \eqref{eq:diagram} commute.
\end{prop}

\begin{proof}
By normality, any such trace $\tau$ is determined by its values on each summand $\Mat_{F\cup\{\al,\om\}}$, and therefore must have the form
\[ \tau(x) = \sum_{F\in\cF} \lambda_F \tr(x_F) \qquad(x\in\cM), \]
where the family $(\lambda_F)_{F\in\cF}$ is strictly positive and absolutely summable with $\sum_{F\in\cF}\lambda_F=1$, and where $\tr$ denotes the \emph{normalized} trace on $\Mat_{F\cup\{\al,\om\}}$. Hence 
\[ \norm{y}_{L^1(\tau)} = \sum_{F\in\cF} \frac{\lambda_F}{\abs{F}+2}\snorm{y_F} \qquad(y\in\cM), \] 
where $\snorm{\cdot}$ denotes the usual, \emph{unnormalized} Schatten $1$-norm on matrices.

It suffices to prove the inequality \eqref{eq:Monty}; the rest is routine.
So, let $a\in c_{00}(\Nat)$. Then
\[ \begin{aligned}
\norm{\phi(a)}_{L^1(\tau)}
  = \sum_{F\in\cF} \frac{\lambda_F}{\abs{F}+2} \snorm{\phi(a)_F} 
 & \leq \sup_{F\in\cF} \frac{1}{\abs{F}+2} \snorm{\phi(a)_F} \\
 & = \sup_{F\in\cF} \frac{1}{\abs{F}+2} \sNorm{\sum\nolimits_{j\in F} a_j E_j} \\
 & \leq \norm{a}_\infty \sup_{F\in\cF} \frac{1}{\abs{F}+2} \sum_{j\in F} \snorm{E_j}
\end{aligned} \] 
Since each $E_j$ is a rank-one operator, $\snorm{E_j}=\norm{E_j}\leq 3$, and the desired inequality \eqref{eq:Monty} now follows.
\end{proof}
\end{section}

\subsection*{Acknowledgments}
The construction in Proposition~\ref{p:extension}, and some of the results from~\cite{YC_acoaf}, formed part of a lecture presented at the conference
{\it Operator theory and applications}, held at Chalmers University, 26th--29th April 2011, in honour of V. S. Shulman. The author thanks the conference organizers for an enjoyable and stimulating meeting.
He would also like to thank Mikael de la Salle and Jesse Peterson for useful discussions on the {\it MathOverflow}\/ website.

Finally, the author thanks the anonymous referee for his or her suggestions and corrections: in particular, the requests for clarification of the original version of Section~\ref{s:prelim}, which led to the repair and improvement of some incomplete arguments.

The work here was partially supported by NSERC Discovery Grant 402153-2011.

\input{bac-oa_bib2}

\vfill\noindent\contact

\noindent
Email: \texttt{choi@math.usask.ca}

\end{document}

%% file: bac-oa_bib2.tex
\newcommand{\rmand}{{\rm and}}
\newcommand{\hurl}[1]{\url{#1}}